\documentclass{amsart}

\usepackage{amssymb}
\usepackage{amsmath}
\usepackage{amsthm}
\usepackage{amsrefs}
\usepackage{comment}
\usepackage{enumitem}

\usepackage[margin=1in]{geometry} 
\usepackage{setspace}
\newcommand{\p}{\mathfrak{p}}
\newcommand{\m}{\mathfrak{m}}
\newcommand{\fq}{\mathfrak{q}}
\newcommand{\fa}{\mathfrak{a}}
\newcommand{\fb}{\mathfrak{b}}

\newcommand\height{\operatorname{ht}}
\newcommand\spec{\operatorname{Spec}}

\newcommand\cal{\mathcal}
\def\frak{\mathfrak}

\newtheorem{thm}{Theorem}[section]

\newtheorem{lemma}[thm]{Lemma}

\theoremstyle{definition}
\newtheorem{definition}[thm]{Definition}
\newtheorem{assumption}[thm]{Assumption}

\newtheorem{example}[thm]{Example}
\newtheorem*{remark}{Remark}
\newtheorem*{question}{Question}

\title{Completely Controlling the Dimensions of Formal Fiber Rings at Prime Ideals of Small Height}
\author{Sarah M. Fleming, Lena Ji, S. Loepp, Peter M. McDonald, Nina Pande, and David Schwein}

\begin{document}
\begin{abstract}
Let $T$ be a complete equicharacteristic local (Noetherian) UFD of dimension $3$ or greater.
Assuming that $|T| = |T/\m|$, where $\m$ is the maximal ideal of $T$,
we construct a local UFD $A$ whose completion is $T$ and whose formal fibers at height one prime ideals
have prescribed dimension between zero and the dimension of the generic formal fiber.
If, in addition, $T$ is regular and has characteristic zero, we can construct $A$ to be excellent.
\end{abstract}
\maketitle

\section{Introduction}
Let $A$ be a local (Noetherian) ring with maximal ideal~$\m$ and $\m$-adic completion~$T$.
To understand the relationship between $\spec A$ and~$\spec T$
it is useful to consider a simplifying numerical quantity,
the dimensions of the formal fiber rings of~$A$,
which roughly measures how much ``larger'' $\spec T$ is than~$\spec A$.
We briefly recall the pertinent definitions.
The \emph{formal fiber ring} of~$A$ at the prime ideal~$\p$
is $T\otimes_A\kappa(\p)$, where $\kappa(\p)$ is the residue field $A_\p\slash\p A_\p$;
the \emph{formal fiber} of~$A$ at~$\p$ is the spectrum of the formal fiber ring of~$A$ at~$\p$;
if~$A$ is an integral domain, the \emph{generic formal fiber ring} of ~$A$ is the formal fiber ring of $A$ at~$(0)$ and the \emph{generic formal fiber} of~$A$ is the formal fiber of $A$ at~$(0)$;
the \emph{dimension of the formal fiber} of~$A$ at~$\p$,
denoted by~$\alpha(A,\p)$, is the dimension
of the formal fiber ring of~$A$ at~$\p$.
The definitions are greatly clarified by 
identifying elements of the formal fibers of~$A$ with prime ideals of~$T$.
This identification uses the natural bijection between
the formal fiber of~$A$ at~$\p$ and
the inverse image of~$\p$ under the map $\spec T\to\spec A$
given by $\fq\mapsto\fq\cap A$.

The first significant study of the dimension of formal fibers
was Matsumura's paper~\cite{matsumura88},
which gives a clear picture of the quantity's basic properties;
we briefly outline some of the most important results from that paper.
Following Matsumura, let $\alpha(A)$
denote the supremum of the dimensions of the formal fibers of $A$.
First, the dimension of the formal fibers decreases
weakly under inclusion, that is,
if $\p$ and $\fq$ are prime ideals of $A$ such that $\p\subseteq\fq$ then $\alpha(A,\p)\geq\alpha(A,\fq)$.
It follows that, if~$A$ is an integral domain, then $\alpha(A)=\alpha(A,(0))$.
Second, if~$A$ has positive Krull dimension~$n$
then its formal fibers have dimension at most~$n-1$,
that is, $\alpha(A)\leq\dim A-1$.
Third, for a large class of naturally occuring rings, the
rings of essentially finite type over a field,
the dimension of the formal fibers depends linearly on height
and decreases as height increases:
$\alpha(A,\p) = \dim A - \height \p - 1$ for $\p$ a nonmaximal prime ideal of $A$.

It is natural to ask whether Matsumura's formula
for~$\alpha(A,\p)$ can be extended to \emph{all} rings,
or at least, to a a class of rings that behave nicely
with respect to completion, such as excellent rings.
We should immediately exclude the trivial case $\alpha(A)=0$, where
every formal fiber has dimension zero.
Besides this trivial case, even if the formula does fail we might reasonably expect
that $\alpha(A)$ would decrease strictly as~$\height\p$ increases.
In particular, simplifying the question by looking at prime ideals of small height,
we might expect that $\alpha(A)$ rarely equals $\alpha(A,\p)$ when $\height\p=1$.

These considerations led Heinzer, Rotthaus, and Sally to informally pose the following question.
\begin{question}
Let A be an excellent local ring with $\alpha(A)>0$, and let $\Delta$ be the set $\{\p\in\spec A\mid\height\p=1\text{ and }\alpha(A,\p) = \alpha(A)\}$.
Is $\Delta$ a finite set?
\end{question}

Answering this question in the negative, Boocher, Daub, and Loepp constructed in~\cite{boocher10} an excellent local unique factorization domain (UFD) with uncountably many height one prime ideals for which $\Delta$ is countably infinite.
Nevertheless, one could still reasonably conjecture that \emph{most} height one prime ideals must lie outside~$\Delta$;
perhaps, for instance, $\Delta$ is at most countable.
It is our goal to disprove this weakened conjecture.
We will construct an excellent UFD~$A$ for which the set~$\Delta$
is uncountable and consists of {\em every} height one prime ideal of~$A$.  Fleming et al.\ show in a forthcoming paper that there exist UFDs 
for which $\Delta$ contains all height one prime ideals, 
but the rings they construct are not excellent.

Our first main result, the outcome of the construction, is Theorem~\ref{mainTheorem}.
We start with a complete equicharacteristic local UFD~$T$ of dimension 3 or greater
having the same cardinality as its residue field.
The theorem then guarantees the existence of a local subring~$A$ of~$T$
which has $T$ as its completion; whose formal fibers at prime ideals of height two or greater are singletons,
and hence have dimension zero;
for which we can prescribe the dimension of the generic formal fiber;
and for which we can prescribe the number of formal fibers
at height one prime ideals that have a given dimension.
In particular, we can force $\alpha(A,\p)$ to equal $\alpha(A,(0))$
for \emph{every} height one prime ideal~$\p$ of~$A$.

Our second main result is Theorem~\ref{excellence}, 
a strengthening of Theorem~\ref{mainTheorem} in which the constructed ring~$A$ is excellent.
For this theorem, we require the additional hypothesis that $T$ be regular and have characteristic zero.

Our method of construction is modeled on the construction in~\cite{heitmann93}:
we build the ring~$A$ by adjoining deliberately chosen elements of~$T$
to the prime subring of $T$.
The key idea, Lemma~\ref{machine}, 
is that the ring~$A$ will be Noetherian and have~$T$ as its completion
provided that the following two conditions hold:
first, the natural map $A\to T/\m^2$ is onto, 
and second, every finitely generated ideal of~$A$
is unchanged by extension to~$T$ followed by contraction to~$A$.
We can construct a ring satisfying the first condition
by starting with the prime subring of~$T$ and adjoining to it
a representative of every coset of~$T/\m^2$.
It turns out that the ring can be made to satisfy
the second condition by adjoining even more elements,
although this step is rather technical; see Lemma~\ref{closeone}.
Since $T/\m^2$ is infinite,
the adjunction process must be formalized using transfinite recursion.
That is, we construct an ascending chain of subrings of~$T$ indexed by ordinals
and terminating in~$A$.
The strength of this method rests in the abundance of viable elements to adjoin.
Broadly speaking, we can force the final ring~$A$ to have a certain property
by choosing the adjoined elements so that the property is created or conserved
during the construction of intermediate rings.
Prime avoidance lemmas function as the main technical tool for making these choices,
and to use these lemmas we must ensure that the intermediate rings 
are always smaller than $T/\m$.
We should caution that although our method is constructive,
it makes uncountably many arbitrary choices
and there is really no hope of explicitly describing the final ring~$A$.

The main ingredient of our construction is a novel combination of properties
and a series of simplifying assumptions
which together yield the ring~$A$ described above.
Before giving the admittedly complicated definition of a \emph{$J$-subring} in Section~3,
which encapsulates these ingredients, let us briefly explain and motivate them.

The height one prime ideals of~$A$ would be simplest to understand
if they were all principal, or equivalently, if $A$ were a UFD.
Since every ring whose completion is a UFD is itself a UFD,
we can ensure that $A$ is a UFD, and hence that its height one prime ideals are simple,
by stipulating that~$T$ be a UFD.
It would be even better if the height one prime ideals of~$A$
were generated by elements that were prime not only in~$A$, but also in~$T$.
We can force $A$ to have this property (which we call \ref{nice6}), that is,
that each of its prime elements be prime in~$T$,
by adjoining to each intermediate ring~$R$
the prime elements of~$T$ that appear as factors of elements of~$R$.
These simplifying assumptions make it much easier to control
the formal fibers of~$A$ in the intermediate rings
because they allow us to restrict 
our attention to the prime elements of~$T$.

To ensure that the formal fibers of~$A$ at its height one prime ideals are large enough,
whenever a generator~$p$ of a height one prime ideal~$pT$ of~$T$ first appears
in an intermediate ring~$R$ we choose a prime ideal~$\fq=\fq_{pT}$
of~$T$, having prescribed height, for which $\fq_{pT}\cap R=pR$.
We call the set of these chosen prime ideals~$\fq$ the \emph{distinguished set for $R$},
denoted by~$\cal Q_R$; the corresponding set of height one prime ideals of~$T$
is denoted by~$\cal F_R$.
To prescribe the height of each~$\fq\in\cal Q_R$,
we first index the elements of $\cal F_R$ by an ordinal number
using an \emph{ordering function}~$\varphi_R$,
and then specify the height of each~$\fq_{pT}$
using a \emph{height indicator} function~$\Lambda$.

The prime ideal~$\fq_{pT}$ will eventually become the maximal element
in the formal fiber of~$A$ at~$pA$.
To ensure that it lies in the fiber at all,
we ensure that $\fq_{pT}\cap R=pR$ holds for every intermediate ring~$R$,
and hence also for~$A$; this step requires prime avoidance lemmas
to carefully choose elements to adjoin.
In this way we get a lower bound of $\height(\fq_{pT}) - 1$
on the dimension of the formal fiber of $A$ at $pA$.
To make this lower bound an upper bound,
we adjoin generators of the prime ideals of~$T$ that are not contained
in any element of the distinguished set.
Consequently, the formal fibers of prime ideals of~$A$
that have height two or more will be singletons.

To control the dimension of the generic formal fiber,
we fix a prime ideal~$\frak P$ of~$T$
having prescribed height and choose the elements to adjoin
so that~$\frak P$ intersects each intermediate ring in the zero ideal.
As with the prime ideals of the distinguished set, we use
$\frak P$ to get a lower bound on the dimension
of the generic formal fiber.
The upper bound on dimension comes from adding nonzero elements of prime 
ideals of $T$ whose height is larger than the height of $\frak P$.  
As long as every prime ideal of~$\height\frak P + 1$ or greater
has nonzero intersection with an intermediate ring,
none of those prime ideals can be in the generic formal fiber.


The previous paragraphs describe the main aspects of the definition
of a $J$-subring. In addition to that definition, we found it convenient
to define an \emph{extension} of a $J$-subring,
two $J$-subrings $R\subseteq S$ satisfying certain compatibility conditions.
Specifically, the rings should have the same size unless $R$ is finite, a trivial case,
and the distinguished sets and ordering functions should be compatible
in the sense that $\cal Q_R\subseteq\cal Q_S$
and $\varphi_S|_{\cal F_R} = \varphi_R$.


In this paper, all rings are commutative with unity.  We say a ring is \emph{quasi-local} if it has exactly one maximal ideal
and \emph{local} if it is both quasi-local and Noetherian.  When we say $(T,\m)$ is a local ring, we mean that $T$ is a local ring with maximal ideal $\m$.
We denote the cardinality of a set $X$ by $|X|$,
the union of the elements of $X$ by $\bigcup X$,
the fraction field of an integral domain~$R$ by $R_{(0)}$,
and the set of prime ideals of a ring~$T$ having height~$k$ by $\spec_k T$.
Several constructions assume familiarity
with the basic properties of ordinal and cardinal numbers. 
Section~2 collects preliminary lemmas,
Section~3 builds the details of the construction,
Section~4 uses the construction to prove our first main result,
and Section~5 extends that result to excellent rings.

\section{Preliminary Lemmas}
The lemmas in this section fall into three broad classes.
First, Lemma~\ref{machine}, our main technical tool;
second, prime avoidance lemmas for choosing the elements to adjoin;
and third, a lemma concerning the cardinality of $T$ and $T/\p$ where $\p$ is a
nonmaximal prime ideal of $T$.

\begin{lemma}[\cite{heitmann94}*{Proposition 1}]
\label{machine}
Let $(A,\m\cap A)$ be a quasi-local subring of a complete local ring $(T,\m)$.
If $\fa T\cap A=\fa$ for every finitely generated ideal~$\fa$ of~$A$
and the map $A\rightarrow T\slash\m^2$ is onto
then $A$ is Noetherian and the natural homomorphism $\widehat A\rightarrow T$ is an isomorphism.
\end{lemma}

Our main technical tool for choosing elements of~$T$
is a pair of powerful prime avoidance lemmas for local rings, Lemmas~\ref{avoidancelem} and~\ref{unittrans}.
We will use the two lemmas to find elements satisfying transcendence conditions
related to the set of distinguished prime ideals of~$T$.
These transcendence conditions guarantee that the contractions
of elements of the distinguished set remain principal
as the intermediate ring grows.

\begin{lemma}[\cite{heitmann93}*{Lemma 3}]
\label{avoidancelem}
Let $(T,\m)$ be a local ring. Let $D\subset T$, let $C\subset\spec T$, 
and let $\fa$ be an ideal of~$T$ such that $\fa\not\subseteq\p$ for all $\p\in C$. 
If $|C\times D|<|T/\m|$ then
\[
\fa\not\subseteq\bigcup\{t+\p\mid\p\in C,t\in D\}.
\]
\end{lemma}

\begin{lemma}[\cite{loepp97}*{Lemma 4}]
\label{unittrans} 
Let $(T,\m)$ be a local ring, with $\left\vert{T\slash\m}\right\vert$ infinite, and let $D\subset T$. Let $C\subset\spec T$ and $y,z\in T$ be such that $y,z\not\in\p$ for all $\p\in C$. If $\left\vert{C\times D}\right\vert<\left\vert{T\slash\m}\right\vert$, then there is a unit $u\in T^{\times}$ such that
\[
yu,zu^{-1}\not\in\bigcup\{t+\p\mid t\in D,\p\in C\}.
\]
\end{lemma}

Because the constructed ring $A$ will have the cardinality of~$T$,
whereas the prime avoidance lemmas are suited to avoiding sets of cardinality $|T/\m|$ or less,
we will require that $|T|=|T/\m|$.
For the same reason it is necessary to constrain the growth
of the size of the intermediate rings, and in particular,
adjoining a single element will not increase the size, provided the ring is already infinite.

The final lemma of this section relates the cardinality of a complete local ring
to that of its quotients, and it shows that complete local rings are ``large."  Note that, a consequence of the lemma is that if $(T,\m)$ is a complete local ring of dimension at least one and $|T/\m| = |T|$, then $|T/\m| > \aleph_0$.

\begin{lemma}[{\cite[Lemma 2.3]{charters04}}]
\label{cardinalitymodp}
Let $(T,\m)$ be a complete local ring of dimension at least one and let $\p\in\spec T$ be a nonmaximal prime ideal of $T$. Then $\left\vert{T\slash\p}\right\vert=\left\vert{T}\right\vert\geq2^{\aleph_0}.$
\end{lemma}

\section{The Construction}

Given a subring~$R$ of a complete local UFD~$T$, let
\[
\cal F_R=\big\{pT\in\spec_1 T\mid pu\in R\text{ for some unit }u\in T\big\}.
\]
For simplicity and without loss of generality
we will subscribe to the notational convention
that the generator~$p$ of a prime ideal~$pT\in\cal F_R$
is contained in~$R$.
That is, when we write $pT\in\cal F_R$ we are choosing $p$ so that $p\in R$.
We could just as well have defined $\cal F_R$ to be a set
of prime elements, but the current definition makes several arguments less verbose.

\begin{definition} Let $(T,\m)$ be a complete local UFD with $3\leq\dim T$,
let $(R,\m\cap R)$ be a quasi-local subring of~$T$ and let
$\mathfrak{P}$ be a nonmaximal prime ideal of~$T$.
Suppose $\Lambda:|T|\to\{1,2,\dots,\min(\height\frak P + 1,\dim T-1)\}$
is a function defined on the cardinal number~$|T|$,
and let $\varphi_R:\cal F_R \to\delta$
be a bijection from $\cal F_R$ to an initial segment~$\delta$ of~$|T|$.  Suppose
that the following conditions are satisfied.

\begin{enumerate}[label=(J.\arabic*)]
	\item\label{nice1} $|R|<|T/\m|$;
    \item\label{nice2} $\mathfrak{P}\cap R=(0)$;
    \item\label{nice3} for every $a\in R$, $aT\cap R=aR$;
    \item\label{nice4} there exists a set $\cal Q_R\subset\bigcup\{\spec_{n} T\mid1\leq n\leq\min(\height\mathfrak{P}+1,\dim T-1)\}$ such that the sets $\cal F_R$ and $\cal Q_R$ are in bijection and the ideal $\fq_{pT}\in\cal Q_R$ corresponding to $pT\in \cal F_R$ is such that
    \begin{enumerate}[label=(J.4\alph*)]
    	\item\label{nice4a} $pT\subseteq\fq_{pT}$,
        \item\label{nice4b} the image of $\fq_{pT}$ in $T/pT$ is in the regular locus of $T/pT$, 
        \item\label{nice4c} $\fq_{pT}\cap R=pT\cap R$, and
         \item\label{nice5} for every $\fq_{pT}\in\cal Q_R$, $\height(\fq_{pT}) = \Lambda(\varphi_R(pT))$.
    \end{enumerate}
\end{enumerate}
Then the $4$-tuple $(R,\frak P,\Lambda,\varphi_R)$ is called a \emph{$J$-subring of~$T$}.  Any set $\cal Q_R$ satisfying \ref{nice4} is called a
\emph{distinguished set for $R$}. 
The function~$\Lambda$ is called the \emph{height indicator} and
the function~$\varphi_R$ is called the \emph{ordering function for $R$}.
If $(R,\frak P,\Lambda,\varphi_R)$ is a $J$-subring of $T$ and $\cal Q_R$ is a fixed distinguished set for $R$, we say that  $(R,\frak P,\Lambda,\varphi_R)$ is a $J$-subring of $T$ with distinguished set $\cal Q_R$. If, in addition, $\frak P$, $\Lambda$, and $\varphi_R$ 
are understood, we will say that  $R$ is a $J$-subring with distinguished set $\cal Q_R$.  If $\cal Q_R$ is also understood, we will say that $R$ is a $J$-subring.
\end{definition}

\begin{remark} 
It follows from conditions~\ref{nice3} and~\ref{nice4c} that $\fq_{pT}\cap R=pT\cap R=pR$.
Note also that no $\fq_{pT}\in\cal Q_R$ is contained in $\frak P$.
\end{remark}

The cardinality condition \ref{nice1} is needed to invoke prime-avoidance lemmas,
and thus allows us to adjoin elements to~$R$ without adding unwanted
elements of~$\mathfrak{P}$ or of prime ideals contained in~$\cal Q_R$.
Although our final ring $A$ will not satisfy~\ref{nice1},
we would like it to satisfy the other properties of a $J$-subring.
\ref{nice2} will ensure that $\mathfrak{P}$ is in the generic formal fiber of $A$.
\ref{nice3} is needed 
to show that the completion of $A$ is $T$.
By maintaining \ref{nice4}, we will ensure that
the elements of $\cal Q_R$ are precisely the maximal ideals of the formal fiber rings of $A$ at its
height one prime ideals;
\ref{nice4b} gives us that these are regular local rings.
The ordering function $\varphi_R$ implicitly identitifes each element~$pT$ of~$\cal F_R$
with an ordinal less than~$|T|$, and the height indicator~$\Lambda$ 
uses this ordinal to specify the height of~$\fq_{pT}$;
this is the content of \ref{nice5}.
We may freely choose the heights of the distinguished prime ideals
as long as each height is strictly less than
that of the maximal ideal and
the dimension of the generic formal fiber is
no less than that of any other fiber;
these two conditions explain the upper bound $\min(\height\frak P + 1,\dim T - 1)$
on the heights of distinguished prime ideals.

\begin{definition}\label{jext}
Let $(T,\m)$ be a complete local UFD with $3\leq\dim T$,
let $(R, \frak P,\Lambda,\varphi_R)$ be a $J$-subring of~$T$ with distinguished set $\cal Q_R$,
and let $(S, \frak P, \Lambda,\varphi_S)$ be 
a $J$-subring of~$T$ with distinguished set $\cal Q_S$ and suppose that $R \subseteq S$.

The ring~$S$ is called a \emph{$J$-extension} of~$R$ if
the following conditions are satisfied:
\begin{enumerate}
\item\label{jext1}
if the ring~$R$ is infinite then~$|S|=|R|$, and
\item\label{jext2}
the distinguished sets and ordering functions are compatible in the sense that
\begin{enumerate}
\item\label{jext2a}
$\cal Q_R\subseteq\cal Q_S$, and
\item\label{jext2b}
$\varphi_R(pT) = \varphi_S(pT)$ for every $pT\in\cal F_R$.
\end{enumerate}
\end{enumerate}
\end{definition}

Condition~\eqref{jext1} ensures that the cardinality
of the rings does not grow too much, and is used only
in the proof of Lemma~\ref{unions}.
Every $J$-subring in this paper will be infinite
except the prime subring of~$T$ in the case when~$T$ has nonzero characteristic.

\begin{lemma}\label{unions}
Let $(T,\m)$ be a complete local UFD with $3\leq\dim T$, 
let $\delta$ be an ordinal number, and
let $(R_\beta, \frak P, \Lambda, \varphi_{R_{\beta}})_{\beta<\delta}$
be a sequence of $J$-subrings of~$T$
indexed by~$\delta$.  Let $\cal Q_{R_{\beta}}$ denote a distinguished set for $R_{\beta}$.
Suppose that 
$R_0$ is infinite; that if
 $\beta=\gamma+1$ is a successor ordinal then $R_\gamma$ is a $J$-extension
of~$R_\beta$;
that if $\beta$ is a limit ordinal then $R_\beta=\bigcup_{\gamma<\beta} R_\gamma$ and
that the distinguished set of~$R_\beta$ is $\cal Q_{R_\beta}=\bigcup_{\gamma<\beta}\cal Q_{R_\gamma}$;
and that the ordering function~$\varphi_{R_\beta}$ is such that
$\varphi_{R_\beta}(pT) = \varphi_{R_\gamma}(pT)$ for every $\gamma < \beta$
and every $pT\in \cal F_{R_\gamma}$.
Define the ring $S$ by $S=\bigcup_{\beta<\delta} R_\beta$.

Then $|S|\leq\max(|R_0|,|\delta|)$;
there is an ordering function $\varphi_S$ such that $(S, \frak P, \Lambda, \varphi_S)$ satisfies every condition of a $J$-subring except perhaps the cardinality condition~\ref{nice1};
and $S$ satisfies condition~\eqref{jext2} of a $J$-extension of~$R_0$. 
\end{lemma}
\begin{proof}
The ring $S$ satisfies~\ref{nice2} since, if $r\in\frak P\cap S$,
then $r\in\frak P\cap R_\beta$ for some $\beta$ and $\frak P \cap R_{\beta} = (0)$.
For~\ref{nice3}, suppose that $a \in S$ and $x \in aT \cap S$.  
There is an $R_\beta$ containing $a$ and $x$,
and since $R_{\beta}$ is a $J$-subring, $aT \cap R_{\beta} = aR_{\beta}$.  
So $x \in aT \cap R_{\beta} = aR_{\beta} \subseteq aS$, and \ref{nice3} follows.

For~\ref{nice4},
note that $\cal F_S=\bigcup_{\beta<\delta}\cal F_{R_\beta}$.  We claim $\cal Q_S = \bigcup_{\beta<\delta} \cal Q_{R_{\beta}}$ is a distinguished set for~$S$.
Conditions~\ref{nice4a}
and~\ref{nice4b} are trivially
satisfied for each element of~$\cal Q_S$.  
Condition~\ref{nice4c} holds because
if $\fq_{pT}\in\cal Q_S$ then
$\fq_{pT}\in\cal Q_{R_\beta}$ for some~$\beta$ and $\fq_{pT}\cap R_{\gamma}=pT\cap R_{\gamma}$ for every $\gamma\geq\beta$;
we may then take the union over~$\gamma$
of each side of these equations to conclude that $\fq_{pT}\cap S=pT\cap S$.
Here we use the coherence condition that $\cal Q_{R_\beta}\subseteq\cal Q_{R_\gamma}$
whenever $\beta\leq\gamma$.

For \ref{nice5}, define the ordering function~$\varphi_S$
on~$S$ by setting $\varphi_S(pT) = \varphi_{R_\beta}(pT)$
for $pT\in\cal F_{R_\beta}$.
It is then obvious that \ref{nice5} and~\eqref{jext2} both hold.

To verify the cardinality inequality, 
we'll prove by transfinite induction the stronger statement that for all~$\beta$,
$|R_\beta|\leq\max(|R_0|,|\beta|)$.
The statement holds trivially for~$\beta=0$.
If $\beta=\gamma+1$ is a successor ordinal
then $|R_\beta| = |R_\gamma|$ because $R_\gamma \subseteq R_\beta$ is a $J$-extension,
and therefore $|R_\beta| \leq \max(|R_0|,|\gamma|) = \max(|R_0|,|\beta|)$.
If $\beta$ is a limit ordinal then
\[
|R_\beta| \leq\sum_{\gamma<\beta}|R_\gamma|
\leq |\beta| \sup_{\gamma<\beta}|R_{\gamma}|
\leq |\beta| \sup_{\gamma<\beta}(\max(|R_0|,|\gamma|))
\leq \max(|R_0|,|\beta|).
\]
It follows that $|S|\leq\max(|R_0|,|\delta|)$.
\end{proof}

Under the assumptions of Lemma~\ref{unions},
if $|R_0|<|T/\m|$ and $|\delta|<|T/\m|$
then $S$ is a $J$-subring of~$T$.
When we apply the lemma
these two inequalities will always hold,
except in the proof of Theorem~\ref{mainTheorem}.

Because many of our lemmas share hypotheses,
we will state them as an assumption
so that they can be concisely referenced.

\begin{assumption}\label{ass1}
The local ring $(T,\m)$ is a complete UFD with $3\leq\dim T$ and $|T/\m|>\aleph_0$, 
$\mathfrak{P}$ is a nonmaximal prime ideal of $T$,
and $(R,\mathfrak{P},\Lambda,\varphi_R)$ is a $J$-subring of~$T$ with distinguished set $\cal Q_R$.
\end{assumption}

Looking at the definition of $J$-subring, it is not immediately clear that finding a set $\cal{Q}_R$ satisfying the conditions of  \ref{nice4} will be possible.  We will employ the next lemma to select the prime ideals $\fq_{pT}$ for $\mathcal{Q}_R$. Note that the conditions in Lemma~\ref{pickq} are precisely the first three conditions of \ref{nice4}.

\begin{lemma}\label{pickq} 
Under Assumption~\ref{ass1}, let $(S,S \cap \m)$ be a quasi-local subring of $T$ such that 
$|S| < |T/\m|$. Let $n$ be an integer with $1\leq n\leq \dim T-1$.
If $pT$ is a prime ideal of $T$
then there exists $\fq_{pT}\in\spec_n T$ satisfying the following:
\begin{enumerate}
\item $pT\subseteq\fq_{pT}$,
\item the image of $\fq_{pT}$ in $T\slash pT$ is in the regular locus of $T\slash pT$, and
\item $\fq_{pT}\cap S=pT\cap S$.
\end{enumerate}
\end{lemma}
\begin{proof}

We show the lemma by induction on $n$. If $n=1$, then $\fq_{pT}=pT$ satisfies the required conditions. 

Now let $1<n\leq \dim T-1$,
and suppose that $\fq\in\spec_{n-1} T$ has the desired properties.
Then $\fq\cap S=pT\cap S$ and so the map
\[
\frac{S}{pT\cap S}\hookrightarrow\frac{T}{\fq}
\]
is an inclusion. 
If $\overline{r}\in S/pT\cap S$ is a nonzero element,
then every height one prime ideal of $T/\fq$ containing $\overline{r}$
is an associated prime ideal of $\overline{r}(T\slash \fq)$.
Thus, since $T/\fq$ is Noetherian,
the set of its height one prime ideals with nonzero intersection with $S/pT\cap S$ is finite if $S/pT \cap S$ is finite, and
has cardinality equal to $|S\slash pT\cap S|\leq|S|<|T/\m|$ otherwise.  Note that $T$ is excellent, and so it is catenary, and it follows that the dimension of $T/\fq$ is at least two.  We next use Lemma \ref{avoidancelem}, on the local ring $(T/\fq, \m/\fq)$ to show that the number of height one prime ideals of $T/\fq$ is at least $|T/\m|$.  Suppose, on the contrary, that the number of height one prime ideals of $T/\fq$ is less than $|T/\m|$, and note that $\m/\fq \not\subset \p/\fq$ for all height one prime ideals $\p/\fq$ of $T/\fq$.  Now let $C = \{\p/\fq \, | \, \p/\fq \mbox{ is a height one prime ideal of } T/\fq\}$, and let $D = \{0 + \fq\}$.  Then we have that $|C \times D| < |T/\m|$, and so by Lemma \ref{avoidancelem}, $\m/\fq \not\subseteq \bigcup \{\p/\fq \, | \, \p/\fq \in C \}$.  But if $x + \fq \in \m/\fq$, then $x + \fq$ is contained in some height one prime ideal of $T/\fq$, a contradiction.  It follows that the number of height one prime ideals of $T/\fq$ is at least $|T/\m|$.
So the set of height one prime ideals of $T/\fq$ whose intersection with $S/pT \cap S$ is the zero ideal is at least $|T/\m|$.
Since $T/pT$ is complete, it is excellent, and it follows that the singular locus of $T/pT$ is closed.  As $T/pT$ is an integral domain, the singular locus is $V(J)$ for some nonzero ideal $J$ of $T/pT$.  Let $J = (g_1,g_2, \ldots, g_t)$.  Letting $\overline{\fq}$ denote the image of $\fq$ in $T/pT$, we have that $\overline{\fq}$ is in the regular locus of $T/pT$, and so $J \not\subset \overline{\fq}$.  Hence, there is a generator of $J$ whose image in $T/\fq$ is not zero.  Without loss of generality, suppose that the image of $g_1$, denoted by $\overline{g_1}$, is not zero in $T/\fq$.  Now choose a height one prime ideal $\overline{\fq_{pT}}$ of $T/\fq$ such that
$\overline{g_1} \not\in \overline{\fq_{pT}}$ and such that the intersection of $\overline{\fq_{pT}}$ with $S/pT \cap S$ is the zero ideal.
As $T$ is excellent, it is catenary, and so $\overline{\fq_{pT}}$ lifts to a prime ideal $\fq_{pT}$ of $T$ of height $n$, and by the way we chose $\overline{\fq_{pT}}$, we have $\fq_{pT} \cap S = pT \cap S$.  If the image of $\fq_{pT}$ in $T/pT$ is in the singular locus of $T/pT$, then it contains $J$, and so, in particular, it contains $g_1$.  But this would imply that $\overline{g_1} \in \overline{\fq_{pT}}$, a contradiction.  It follows that the image of $\fq_{pT}$ in $T/pT$ is in the regular locus of $T/pT$.
\end{proof}


The next lemma allows us to adjoin a transcendental element
to a $J$-subring of $T$
and find a $J$-extension
maintaining the properties that we desire.
Because we repeatedly need to adjoin transcendental elements to satisfy different aspects of our construction, this lemma will be used often in proofs of later lemmas.

\begin{lemma} \label{transextension} 
Under Assumption~\ref{ass1}, 
suppose that $x\in T$ is such that $x+\fq \in T/\fq$ is transcendental over $R\slash\fq\cap R$ for all $\fq\in\cal{Q}_R\cup\{\mathfrak{P}\}$.
Then there exists an infinite $J$-extension~$S$ of~$R$ that contains~$x$.
\end{lemma}
\begin{proof} 
Let $S=R[x]_{(0)}\cap T$.
It is readily seen that $S$ is infinite, that $S$ satisfies~\ref{nice1} and, if $R$ is infinite, that $|R|=|S|$.
Additionally, if $c\in aT\cap S$, then $c=ab$ for some $b\in S_{(0)}\cap T=S$, and 
so $c\in aS$ and condition~\ref{nice3} also holds.
For~\ref{nice2} consider any element $r\in\frak P\cap S$.
Then $r=f/g$ for some $f\in\frak P\cap R[x]$.
Treating $f$ as a polynomial in $x$ over $R$,
the assumption that $x+\frak P$ is transcendental over $R/\frak P\cap R$,
gives us that each of its coefficients must be in $\frak P\cap R=(0)$.
Hence $r=0$ and we have that $\frak P\cap S=(0)$.
Now most of the work of showing that $S$ is a $J$-extension of $R$ consists of
constructing the distinguished set~$\cal Q_S$ for~$S$ and verifying that it satisfies the necessary properties.

Using the axiom of choice, find an arbitrary 
extension of the ordering function~$\varphi_R$ 
to an ordering function~$\varphi_S$.
To define~$\cal Q_S$, select for each $p'T\in\cal F_S-\cal F_R$ an ideal~$\fq_{p'T}$
such that $\height(\fq_{p'T}) = \Lambda(\varphi_S(p'T))$, using Lemma~\ref{pickq};
then \ref{nice5} holds and each~$\fq_{p'T}$ satisfies condition~\ref{nice4}.
Define $\cal Q_S$ to be the union of $\cal Q_R$ and the set of recently chosen $\fq_{p'T}$.
To finish showing that $S$ is our desired $J$-subring, it remains to check condition \ref{nice4c}
for $\fq_{pT}\in\cal Q_R$.

We first show that $\fq_{pT}\cap R[x]=pT\cap R[x]=pR[x]$ for every $\fq_{pT}\in\cal Q_R$. 
The assumption that $x+\fq_{pT}$ is transcendental over $R/(\fq_{pT} \cap R)$
implies that if an element of $R[x]$ is contained in $\fq_{pT}$, 
then each of its coefficients is in $\fq_{pT}\cap R=pT\cap R=pR$. 
Thus $\fq_{pT}\cap R[x]\subseteq(\fq_{pT}\cap R)R[x]=pR[x]\subseteq pT\cap R[x]$,
and since $pT\subset\fq_{pT}$, we have equality throughout.

To now verify condition~\ref{nice4c} for $pT\in\cal F_R$, 
let $r=f/g\in\fq_{pT}\cap R[x]_{(0)}$ with $f,g\in R[x]$.
If $f$ and $g$ are both contained in $\fq_{pT}\cap R[x]=p R[x]$,
then we can write $f=p^n f'$ and $g=p^m g'$
with $f',g'\not\in\fq_{pT}$
and $r=(p^n f')/(p^m g')$.
Thus, by dividing out the highest common power of $p$
we can assume that $p$
does not divide every coefficient of~$f$ and~$g$,
and therefore that at most one of~$f$ or~$g$ lies in~$\fq_{pT}$.
Since $f=rg\in\fq_{pT}$, we have that $g\notin\fq_{pT}$.
Thus we can write $f=pf'$ for $f'\in R[x]$, and
unique factorization gives that $f'/g\in T$. 
Therefore $r\in pT$.
Intersecting both sides of $pT\subseteq\fq_{pT}$ with $S$
gives the opposite inequality, and so $\fq_{pT}\cap S=pT\cap S$.
\end{proof}

As noted above, there are various aspects of the construction 
that require us to adjoin transcendental elements to our ring $R$.  
The following lemmas cover these scenarios.

Recall from Lemma~\ref{machine} that in order for 
the constructed ring $A$ to have completion $T$,
the map $A\rightarrow T/\m^2$ needs to be surjective.
To guarantee surjectivity,
we will use Lemma~\ref{cosets} to adjoin
coset representatives of elements of $T/\m^2$.
The proof of this lemma resembles the proof of Lemma 2.5 of~\cite{charters04}.
We also need to ensure that $\fa T\cap A=\fa$ for every finitely generated ideal $\fa$ of $A$;  
Lemmas~\ref{closeone} and \ref{ideals} will help us obtain this property.

\begin{lemma}\label{cosets} 
Under Assumption~\ref{ass1}, let $t\in T$.
Then there exists a $J$-extension $S$ of $R$ such that $t+\m^2$ is in the image of the map $S\rightarrow T\slash\m^2$.
\end{lemma}
\begin{proof}

For each $\fq\in\cal Q_R\cup\{\mathfrak{P}\}$,
let $D_{\fq}$ be a full set of representatives of
the cosets $x+\fq \in T/\fq$ such that
$t+x+\fq \in T/\fq$ is algebraic over $R\slash\fq\cap R$;
then, if $R$ is infinite, $|D_{\fq}| \leq |R|<|T\slash\m|.$  If $R$ is finite, then it is clear that $|D_{\fq}|<|T\slash\m|.$
Define $D=\bigcup\{D_\fq\mid\fq\in\cal Q_R\cup\{\mathfrak{P}\}\}$.
As $\m^2\not\in\fq$ for any $\fq\in\cal{Q}_R\cup\{\frak P\}$
and $|D|<|T\slash\m|$,
we can choose an element $x\in\m^2$ using
Lemma $\ref{avoidancelem}$ such that $x+t+\fq$ is transcendental over $R\slash\fq\cap R$ for all $\fq\in\cal{Q}_R\cup\{\mathfrak{P}\}$.
By Lemma $\ref{transextension}$ we can find a
$J$-extension $S$ of~$R$ containing $t+x$. 
Then $t+\m^2$ is in the image of the map $S\rightarrow T/\m^2$.
\end{proof}

\begin{lemma} \label{closeone} 
Under Assumption~\ref{ass1},
let $\fa$ be a finitely generated ideal of $R$ and let $c\in\fa T\cap R$.
Then there exists a $J$-extension~$S$ of $R$ such that $c\in\fa S$.
\end{lemma}
\begin{proof}
Our proof is similar to that of Lemma~4 of~\cite{heitmann93}, and
we proceed by induction on the number of generators of~$\fa$.  It will be clear from the construction that, if $R$ is infinite, then $|R|=|S|$.
If $\fa$ is the zero ideal, then $S = R$ works.  So, for the rest of the proof, assume that $\fa$ is not the zero ideal.
Note that, if $n=1$, then $S = R$ works since $R$ is a $J$-subring.

For the $n=2$ case, let $\fa$ be generated by $a_1$ and $a_2$
and write $c = a_1x_1 + a_2x_2$ for some $x_1,x_2\in T$. 
Suppose first that $a_1,a_2 \in \fq_{pT}$ for some $\fq_{pT}\in\cal Q_R$
with $p\in R$. Then $a_1$ and $a_2$ are both in $\fq_{pT}\cap R=pR$,
so by dividing out the highest power of~$p$ dividing both~$a_1$
and~$a_2$, we can replace these with $a_1'$ and $a_2'$
where at least one of $a_1'$ and $a_2'$ is not in $\fq_{pT}$.
To prove the lemma, it suffices to prove it with $c$ replaced by the element $c'=a_1'x_1+a_2'x_2$, 
$a_1$ replaced by $a'_1$, and $a_2$ replaced by $a'_2$, but we now have that at least one of $a'_1$ and $a'_2$ is not contained in $\fq_{pT}$.
By repeating this process if necessary we may assume without loss
of generality that, if $\fq\in\cal Q_R$, then $a_1$ and $a_2$ are not both in $\fq$.  Note also that neither $a_1$ nor $a_2$ are in $\mathfrak{P}$.

Let $x_1'=x_1+a_2y$ and $x_2'=x_2-a_1y$ for some $y\in T$ to be chosen soon
using Lemma~\ref{avoidancelem}.
For each $\fq\in\cal Q_R\cup\{\mathfrak{P}\}$ such that $a_2\notin\fq$,
let $D^1_\fq$ be a set of representatives for the cosets
$y+\fq$ such that $x_1'+\fq$ is algebraic over $R/\fq\cap R$.
If $a_2\in\fq$, let $D^1_\fq=\varnothing$.
Define $D^2_\fq$ similarly with $a_2$ replaced by $a_1$ and $x_1'$ replaced by $x_2'$, and let
$D=\bigcup\{D^1_\fq\cup D^2_\fq\mid\fq\in\cal Q_R\cup\{\mathfrak{P}\}\}$.
Use Lemma~\ref{avoidancelem} to choose an element $y\in T$ such that
$y\notin\bigcup\{t+\frak q\mid t\in D,\frak q\in\cal Q_R\cup\{\mathfrak{P}\}\}$.
Let $S=R[x_1']_{(0)}\cap R[x_2']_{(0)}\cap T$;
we claim that $S$ is the desired $J$-extension.
Note that \ref{nice2} is satisfied, since the fact that $x_1'+\frak P$ is transcendental over $R/\frak P\cap R$
implies that $\frak P\cap R[x_1']_{(0)}=(0)$, as in the proof of Lemma~\ref{transextension},
and so $\frak P\cap S=(0)$.
Since $c\in\fa S$
and since $S$ satisfies~\ref{nice1} and since it is easy to verify, as in the proof of Lemma \ref{transextension}, that $S$ satisfies~\ref{nice3}, 
it remains to extend the distinguished set of~$R$
to~$S$ and verify that the necessary properties hold. 
Let $\fq_{pT}\in\cal Q_R$;
at least one of $a_1$ or $a_2$ is not in $\fq_{pT}$,
so without loss of generality we may assume $a_2\notin\fq_{pT}$.
Properties \ref{nice4a} and \ref{nice4b} still hold, and, as in the proof of 
Lemma~\ref{transextension}, we can conclude that
$\fq_{pT}\cap R[x_1']_{(0)}=pT\cap R[x_1']_{(0)}$.
Intersecting this equation with $S$
gives $\fq_{pT}\cap S=pT\cap S$, as needed.  Using the axiom of choice, find an
arbitrary extension of the ordering function $\varphi_R$ to an ordering function $\varphi_S$.
Finally, for elements $pT$ of $\cal F_S-\cal F_R$, use Lemma~\ref{pickq}
to choose $\fq_{pT}\in\spec_dT$ where $d = \Lambda(\varphi_S(pT))$
so that these prime ideals satisfy
the conditions listed in~\ref{nice4}.
Defining $\cal Q_S$ to be the union of
$\cal Q_R$ with the set of these $\fq_{pT}$
completes the proof of the case $n=2$.

For the general inductive step,
let $\fa=(a_1,\dots,a_n)$ be an $n$-generated ideal of~$R$
and let $c=\sum_{i=1}^n a_ix_i\in\fa T\cap R$
for $x_i\in T$. As in the $n=2$ case, we may assume
without loss of generality that if $\fq \in \cal Q_R$, then at least one of $a_1,a_2, \ldots ,a_n$ is not contained in $\fq$.
Let $\fb=(a_1,\dots,a_{n-1})R$.
The inductive step breaks into two cases,
depending on whether or not there is a $\fq \in \cal Q_R$ that contains all of $a_1, a_2, \ldots ,a_{n - 1}$.

Suppose first that there is no such $\fq$ in $\cal Q_R$.
Let $x_n'=x_n + a_1y_1 + \cdots + a_{n-1}y_{n-1}$
and $x_i'=x_i-a_ny_i$ ($i\neq n$) for some $y_i\in T$
to be determined, so that $c = \sum_{i=1}^n a_ix_i'$.
Using Lemma~\ref{avoidancelem}, 
choose $y_1$ so that $x_n + a_1y_1 + \fq$
is transcendental over $R/\fq\cap R$ for every $\fq\in\cal Q_R\cup\{\mathfrak{P}\}$
such that $a_1\notin\fq$.
Next, choose $y_2$ so that $x_n + a_1y_1 + a_2y_2 + \fq$
is transcendental over $R/\fq\cap R$
for every $\fq\in\cal Q_R\cup\{\mathfrak{P}\}$ such that $a_2\notin\fq$.
This choice will not affect the previous transcendence conditions:
that is, if $x_n + a_1y_1 + \fq$ was transcendental over
$R/\fq\cap R$ then $x_n + a_1y_1 + a_2y_2 + \fq$ will
also be transcendental.
Continuing to choose $y_i$'s in this way,
the final element $x_n'$ will have the property that
$x_n'+\fq$ is transcendental over $R/\fq\cap R$
for every $\fq\in\cal Q_R\cup\{\mathfrak{P}\}$ not containing
each of $a_1,a_2,\dots,a_{n-1}$.
By our earlier assumption, 
there is no $\fq\in\cal Q_R$ that contains
all of these elements.
We may therefore use Lemma~\ref{transextension}
to procure a $J$-extension $R'$
of~$R$ containing~$x_n'$.  Letting $c' = \sum_{i = 1}^{n - 1} a_ix_i'$, we have $c' \in \fb T \cap R'$ and so, by induction, there is a $J$-extension $S$ of $R'$ such that $c' \in \fb S$.  It follows that $c \in \fa S$.

Suppose next that there is a $\fq \in \cal Q_R$ such that $a_1,a_2, \ldots ,a_{n - 1}$ are all contained in $\fq$.
This final step of the proof is the most elaborate, and will proceed in part by reduction to previous cases.
The idea of the proof that follows
is to replace $\fb$ with an ideal~$\fb'$ whose generators share no common factor in~$R$.

Factoring out common divisors of the $a_i$'s as in the $n=2$ case,
we can find an element~$r\in R$ with
$a_i=r a_i'$ for each $i<n$
and such that, for every $\fq \in \cal Q_R$, at least one of $a'_1, a'_2 \ldots ,a'_{n - 1}$ is not in $\fq$.  Let $\fb'=(a_1',\dots,a_{n-1}')$.
Write $c$ as $c = rx + a_nx_n$,
where $x=a_1'x_1 + \cdots + a_{n-1}'x_{n-1}$.
We can now apply the construction of the $n=2$ case
to the ideal $(r,a_n)$
to find a $J$-extension $R'$ of~$R$
containing elements $x'$ and $x_n'$
such that $c=rx'+a_nx_n'$.
To finish, we'll find a sufficient
condition on a ring~$S$ containing~$R'$ so that $c\in\fa S$,
then construct~$S$ to satisfy the condition.
In general, to ensure that $c\in\fa S$ it is enough
that $rx'\in\fa S$, and to ensure that $rx'\in\fa S$ it is enough
that $x'\in\fb'S + a_n S$.
By the method of construction in the $n=2$ case, the element $x'$
has the form $x' = x + ta_n$ for some $t\in T$;
hence $x'\in(\fb' T + a_n T)\cap R'$.
In this way the problem is reduced to constructing
a $J$-extension $S$ of $R'$
such that $x'\in (\fb'S + a_n S)$.

If, for every $\fq \in \cal Q_{R'}$, at least one of $a'_1, a'_2, \ldots , a'_{n - 1}$ is not contained in $\fq$, then
we can finish by reducing to the previous case.
However, even though for every $\fq \in \cal Q_R$, at least one of $a'_1, a'_2, \ldots , a'_{n - 1}$ is not contained in $\fq$, 
it might happen that there is a $\fq \in \cal Q_{R'}$ containing all of $a'_1,a'_2 \ldots ,a'_{n - 1}$, since $R'$ might contain primes of~$T$ that divide elements of~$R$ but that were not present in~$R$.
If so, repeat the argument of the previous paragraph to modify the generators of $\fb'$.  We know this
process must end since the generators of $\fb'$ have only finitely many prime factors in $T$.  When the process ends, we may reduce to the case where the generators of $\fb'$ are not all contained in an element of~$\cal Q_{R}$.
\end{proof}

At the end of our construction,
if $pT\in\spec_1T$ is not in the generic formal fiber of $A$,
then it must be contained in the formal fiber of some height one prime ideal of $A$.
In order for the formal fiber ring of this prime ideal to be regular it must be reduced, and 
so we would like $pT$ to be the unique minimal element of that formal fiber.
We can achieve this if $A$ has the property that if $pT\in\spec_1 T$ has nonzero intersection with $A$, then $pT\in\cal F_A$.  We therefore introduce the following property for a subring $R$ of $T$.
\begin{enumerate}[label=(J.\arabic*)]
\setcounter{enumi}{4}
	\item\label{nice6} if $pT\in\spec_1 T$ has nonzero intersection with $R$, then $pT\in\cal F_R$.
\end{enumerate}

Thus, as we continue our construction,
we are particularly interested in $J$-subrings of $T$ satisfying~\ref{nice6}.
The next lemma allows us to find such subrings.

\begin{lemma} \label{ideals}
Under Assumption~\ref{ass1}, there exists a $J$-extension~$S$ of~$R$
such that, for every finitely generated ideal~$\fa$
of~$S$, $\fa T\cap S = \fa$ and $S$ satisfies \ref{nice6}
\end{lemma}
\begin{proof}  First suppose that $R$ is finite.  Let $C = \cal Q_R \cup \{\frak P\}$, and, for each $\fq \in C$, let $D_{\fq}$ be a full set of coset representatives of the cosets $x + \fq \in T/\fq$ that are algebraic over $R/R \cap \fq$.  Now define $D = \cup\{D_{\fq} \mid \fq \in C \}$ and, as in previous proofs, use Lemma \ref{avoidancelem} to find an element $x \in T$ such that $x + \fq$ is transcendental over $R/R \cap \fq$ for all $\fq \in C$.  By Lemma \ref{transextension}, there is a $J$-extension $R_0$ of $R$ that contains $x$.  Note that $R_0$ is infinite.  We will construct $S$ to be a $J$-extension of $R_0$, and so it will also be a $J$-extension of $R$.  If $R$ is infinite, then let $R_0 = R$.

Now we inductively construct a countable ascending chain $R_0\subseteq R_1\subseteq\cdots$ of
$J$-extensions and take $S=\bigcup_{i = 1}^{\infty} R_i$.
The rings will be constructed so that,
if $\fa$ is a finitely generated ideal of~$R_n$,
then $\fa T\cap R_n\subseteq \fa R_{n+1}$ and $\{pT\in\spec_1 T\mid pT\cap R_n\neq(0)\}\subseteq\mathcal F_{R_{n+1}}$ for each~$n$.
We may then take $S=\bigcup_{i = 1}^{\infty} R_i$. 
By Lemma~\ref{unions} the ring~$S$ is a $J$-extension of $R$,
and since $\cal F_S=\bigcup_{i = 1}^{\infty}\cal F_{R_i}$, the ring~$S$ satisfies~\ref{nice6}. If $\fa$ is a finitely generated ideal of~$S$ and $c\in\fa T\cap S$, then there is an $n$ such that the generators of $\fa$ are in $R_n$ and $c\in\fa T\cap R_n$, and hence $c\in \fa R_{n+1}\subseteq \fa S = \fa$.

It remains to construct the countable ascending chain.
Having already defined $R_0$,
suppose inductively that the ring~$R_n$ has been constructed.
To construct $R_{n+1}$, we proceed in two steps.  First we construct $R_{n+0.5}$ by constructing a second ascending chain of rings
as follows.
Let $\Omega$ denote the set of tuples $(\fa,c)$,
where $\fa$ is a finitely generated ideal of~$R_n$
and $c\in\fa T\cap R_n$.
We have that $|\Omega| = |R_n|$ because $R_n$ is infinite and so it has the same cardinality as the set of its finite subsets.
Well order $\Omega$ so that it does not have a maximal element, write $\lambda=|\Omega|$, and let $(\fa_\beta,c_\beta)$ denote the element of~$\Omega$ corresponding to the ordinal $\beta\in\lambda$.
We construct a chain~$(R^\beta)_{\beta\in\lambda}$, and we
let $R^0=R_n$ to start.
If $\beta=\gamma+1$ is a successor ordinal,
let $R^\beta$ be a $J$-extension of~$R^\gamma$
such that $c_\gamma\in \fa_\gamma R^\gamma$,
obtained by Lemma~\ref{closeone}.
If $\beta$ is a limit ordinal,
let $R^\beta=\bigcup_{\gamma<\beta} R^\gamma$ and define $\cal Q_{R_{\beta}} = \bigcup_{\gamma < \beta} \cal Q_{R_{\gamma}}$.  We claim that the chain $(R^\beta)_{\beta\in\lambda}$ satisfies the hypotheses of Lemma \ref{unions}.  To show this, we show that $R^{\beta}$ satisfies the conditions of Lemma \ref{unions} for all $\beta \in \lambda$.  Note that $R^0 = R_n$ is a $J$-subring, and so it satisfies the conditions of Lemma \ref{unions}.  Now suppose that $R^{\gamma}$ satisfies the conditions of Lemma \ref{unions} for all $\gamma < \beta$.  If $\beta$ is a successor ordinal, then $R^{\beta}$ is a $J$-extension of $R^{\gamma}$ by definition.  If $\beta$ is a limit ordinal, then Lemma \ref{unions} gives us that $R^{\beta}$ is a $J$-subring and in the proof of Lemma \ref{unions} we see that the distinguished set for $R^{\beta}$ is $\cal Q_{R_{\beta}} = \bigcup_{\gamma < \beta} \cal Q_{R_{\gamma}}$.  We also see in the proof that the ordering function for $R^{\beta}$ satisfies $\varphi_{R_{\beta}}(pT) = \varphi_{R_{\gamma}}(pT)$ for every $\gamma < \beta$ and every $pT \in \cal F_{R{\gamma}}$.  It follows that $R^{\beta}$ satisfies the conditions of Lemma \ref{unions}.
Finally, let $R_{n+0.5}=\bigcup_{\beta\in\lambda} R^\beta$. By Lemma \ref{unions}, $R_{n+0.5}$ is a $J$-extension of $R_n$.  It is clear by construction that 
if $\fa$ is a finitely generated ideal of~$R_n$ then $\fa T\cap R_n\subseteq\fa R_{n+0.5}$.

Next we proceed from $R_{n+0.5}$ to $R_{n+1}$ in a similar fashion, constructing another ascending chain of $J$-subrings of $T$.  Let $$\cal G_{R_n}=\{pT\in\spec_1 T\mid pT\cap R_n\neq(0)\} $$
and let $\mu=|\cal G_{R_n}|$.
We will construct a chain of $J$-subrings $(R^{\beta})_{\beta \in \mu}$ so that they all contain $R_n$.  It follows that if $R^{\gamma}$ is an element of our chain, and $pT \in \cal G_{R_n}$, then $pT \cap R^{\gamma} \neq (0)$.
Well-order the set~$\cal G_{R_n}$ so that it does not have a maximal element
and let $p_\beta$ denote a generator of an element of $\cal G_{R_n}$
corresponding to the ordinal~$\beta$.
To start, set $R^0=R_{n+0.5}$.
If $\beta=\gamma+1$ is a successor ordinal, and $p_{\gamma}T \in \cal F_{R^{\gamma}}$ then define $R^{\beta}$ to be $R^{\gamma}$.
If $p_{\gamma}T \not\in \cal F_{R^{\gamma}}$, and $p_{\gamma} \in \fq_{pT}$ for some $\fq_{pT} \in \cal Q_{R^{\gamma}}$ then let $r \in p_{\gamma}T \cap R^{\gamma}$.   So $r = p_{\gamma}t$ for some $t \in T$ and $r = p_{\gamma} t \in R^{\gamma} \cap \fq_{pT} = pR^{\gamma}$, and we have $r = p_{\gamma}t = pr_1$ for some $r_1 \in R^{\gamma}$.  Since $p_{\gamma}T \not\in \cal F_{R^{\gamma}}$, $p_{\gamma}T \neq pT$.  It follows that $p$ divides $t$.  Dividing out by $p$, we get $p_{\gamma}t' = r_1 \in \cal \fq_{pT} \cap R^{\gamma} = pR^{\gamma}$ for some $t' \in T$.  Hence $p$ divides $t'$.  Continue in this way to show that $r \in p^iT$ for all $i = 1,2,\ldots$.  Since $\bigcap_{i = 1}^{\infty} p^i T = (0)$ we have that $r = 0$, contradicting that $p_{\gamma}T \cap R_n \neq (0)$.  So if $\fq \in \cal Q_{R^{\gamma}}$ then $p_{\gamma} \not\in \fq$.
Now use Lemma~\ref{unittrans}
to find a unit $u\in T$ such that
$up_\gamma + \fq$ is transcendental over $R^{\gamma}/\fq\cap R^{\gamma}$
for every $\fq\in\cal Q_{R^{\gamma}} \cup \{\frak{P}\}$;
then use Lemma~\ref{transextension} to construct $R^\beta$
to be a $J$-extension of~$R^\gamma$
containing $up_\gamma$.
If $\beta$ is a limit ordinal,
let $R^\beta=\bigcup_{\gamma<\beta}R^\gamma$.
Finally, let $R_{n+1}=\bigcup_{\beta\in\mu} R^\beta$.
Following the proof from the previous paragraph, we can conclude that $R_{n + 1}$ is a $J$-extension of $R_{n + 0.5}$, and so $R_{n + 1}$ is a $J$-extension of $R_n$.
It is clear from the construction that $\cal G_{R_n}\subseteq\cal F_{R_{n+1}}$.  Furthermore, we note that if $\fa$ is a finitely generated ideal of $R_n$, $\fa T\cap R_n\subseteq\fa R_{n+0.5}\subseteq \fa R_{n+1}$.
\end{proof}

Recall that the ring $A$ we construct will have the property that
the formal fiber ring of $\p\in\spec A$ is a field
if $\height\p>1$.
To achieve this, we would like $\p T$ to be the only element
in the formal fiber of $A$ at $\p$;
applying Lemma~\ref{generators} to certain prime ideals of $T$
will allow us to accomplish this.
The same lemma also enables us to add nonzero elements
from prime ideals of $T$ with height greater than that of $\frak P$
which will ensure that $\alpha(A,(0))$ is no greater than $\height\frak P$.

The following lemma is used in the proof of \ref{generators} and comes from the proof of Lemma 3 in~\cite{heitmann93}.

\begin{lemma}\label{avoidvs} Let $k$ be an infinite field, let $V$ be a finite-dimensional $k$-vector space,
and let $\{V_\beta\}_{\beta\in A}$ be a collection of $k$-vector spaces
indexed by a set~$A$ with $|A|<|k|$.
If for all $\beta\in A$ we have $V\not\subseteq V_\beta$,
then $V\not\subseteq\bigcup_{\beta\in A}V_\beta$.
\end{lemma}

\begin{lemma}\label{generators} 
Under Assumption~\ref{ass1}, suppose that $T$ contains a field~$k$ with $|k|=|T|$.
Let $\fa$ be an ideal of $T$ such that $\fa \not\subseteq \frak P$ and such that, if $\fq \in\cal Q_R$, then
~$\fa \not\subseteq \fq$.
Then there is a $J$-extension $S$ of~$R$ such that either $\fa \subseteq \p$ for some $\p \in \cal Q_S$ or $(\fa\cap S)T=\fa$.
\end{lemma}
\begin{proof}
Let $a_1,\dots,a_n$ be a minimal set of generators for~$\fa$,
and let $V$ denote the $k$-vector space generated by the $a_i$.
We will inductively construct a basis $\{b_1,\dots,b_n\}$ for~$V$
and a chain of $J$-extensions $R=R_0\subset R_1\subset\cdots\subset R_n$.  If, for some $i$, we have $\fa \subseteq \p$ for some $\p \in \cal Q_{R_i}$, we stop the construction of the basis and declare that $S = R_i$.  If, on the other hand, this does not happen, then
the $R_i$ will be chosen so that, for every $i$, there is a unit $u_i\in T$
such that $u_ib_i\in R_i$.
It then follows that for $S = R_n$, we have $(\fa\cap S)T=\fa$.

Suppose now that $R_i$ has been constructed and assume that, if $\fq \in \cal Q_{R_i}$, then $\fa \not\subseteq \fq$.
To construct $R_{i+1}$,
note that each element of~$\cal Q_{R_i} \cup \{ \frak P\}$ is a $k$-vector space and
apply Lemma~\ref{avoidvs} to the set $\cal Q_{R_i}\cup \{ \frak P \} \cup \{W_i\}$,
where $W_i$ is the proper subspace of~$V$ spanned by $b_1,\dots,b_i$.
In this way we obtain an element $b_{i+1}\in V$ that is not contained in any element of~$\cal Q_{R_i}\cup\{\mathfrak{P}\}$ and that is not in the span of $b_1,\dots,b_i$.  

Next, we'll use Lemma~\ref{unittrans} to choose $u_{i+1}$.
For each $\fq\in\cal Q_{R_i}\cup\{\mathfrak{P}\}$, let $D_\fq$ be a set of representatives
for the cosets of $\fq$ that are algebraic over~$R_i/\fq\cap R_i$.
Let $D=\bigcup\{D_\fq\mid\fq\in\cal Q_{R_i}\cup\{\mathfrak{P}\}\}$, and choose $u_{i+1}$
so that
\[	
u_{i+1}b_{i+1}\notin\bigcup\{t+\frak q\mid t\in D, \frak q\in\cal Q_{R_i}\cup\{\mathfrak{P}\}\}.
\]
Finally, let $R_{i+1}$ be a $J$-extension of~$R_i$
containing $u_{i+1}b_{i+1}$, obtained by Lemma~\ref{transextension}.  If $\fa \subseteq \p$ for some $\p \in \cal Q_{R_{i + 1}}$, then let $S = R_{i + 1}$.  Otherwise, continue to construct $R_{i + 2}$.  It follows that either $\fa \subseteq \p$ for some $\p \in \cal Q_S$ or $(\fa\cap S)T=\fa$.
\end{proof}

\begin{lemma}\label{pretheorem}
Under Assumption \ref{ass1}, 
let $t\in T$, let $\fa$ be an ideal of~$T$ not contained
in $\frak P$ or in any element of~$\cal Q_R$,
and suppose that $T$ contains a field $k$ with $|k|=|T|$.
Then there exists a $J$-extension $S$ of $R$ such that
\begin{enumerate}
\item
$t+\frak m^2$ is in the image of the natural map $S\rightarrow T/\m^2$,
\item
Either $\fa \subseteq \p$ for some $\p \in \cal Q_S$ or $(\fa\cap S)T=\fa$,
\item
for every finitely generated ideal $\fb$ of $S$ we have that $\fb T\cap S=\fb$, and
\item
$S$ satisfies \ref{nice6}
\end{enumerate}
\end{lemma}
\begin{proof}
Use Lemma~\ref{generators} to find a $J$-extension $R'$ of $R$ such that either $\fa \subseteq \p$ for some $\p \in \cal Q_{R'}$ or $(\fa\cap R')T=\fa$.  
Then use Lemma~\ref{cosets} to find a $J$-extension $R''$ of $R'$ 
such that $t+\m^2$ is in the image of the map $R''\rightarrow T/\m^2$. 
Finally, use Lemma~\ref{ideals} to find a $J$-extension~$S$ of~$R''$ satisfying \ref{nice6} such that if $\fb$ is a finitely generated ideal of~$S$ then $\fb T\cap S=\fb S$.
\end{proof}

\section{The Main Theorem}

\begin{thm}\label{mainTheorem} 
Let $T$ be a complete equicharacteristic local UFD with $3\leq\dim T$ and $|T|=|T/\m|$.
Let $\frak P$ be a nonmaximal prime ideal of~$T$
and let $\{\lambda_d\mid 0\leq d\leq\min(\height\frak P,\dim T - 2)\}$ be a set of cardinal numbers such that
$\sum_d \lambda_d = |T|$.

Then there exists a local UFD $A$ such that $\widehat A = T$ and the following are satisfied:
\begin{enumerate}
\item\label{main1}
$\frak P\cap A=(0)$ and
$\alpha(A,(0))=\height\mathfrak{P}$,
\item\label{main2}
for each $d$ the set $\Delta_d=\{pA\in\spec_1 A\mid\alpha(A,pA)=d\}$
has cardinality $|\Delta_d|=\lambda_d$,
\item\label{main3} 
the ring $T\otimes_A\kappa(pA)$ is regular local for every $pA\in\spec_1 A$, and
\item\label{main4}
the ring $T\otimes_A\kappa(\p)$ is a field for every $\p\in\spec A$ with $\height\p\geq 2$.
\end{enumerate}
\end{thm}
\begin{proof}
We will use transfinite recursion
to construct~$A$ as the union of an ascending chain 
$(R_\beta)_{\beta<|T|}$ of extensions of the prime subring of~$T$.
To make the construction easier to understand, we begin with an overview.

The main task is to ensure that $A$ is Noetherian and has the correct completion, using Lemma~\ref{machine}.
To satisfy the first hypothesis of the lemma, that the map $A\to T/\m^2$ is onto, we form $R_{\beta+1}$ from $R_\beta$ by adjoining a coset representative for an element of~$T/\m^2$, eventually adjoing a representative for every coset.
To satisfy the second hypothesis of the lemma, that finitely generated ideals are unchanged by extension to~$T$ and contraction to the original ring, we construct 
$R_{\beta+1}$ so that it satisfies this property.
Ensuring that~$A$ is a UFD requires no work since in general any local domain whose completion is a UFD is itself a UFD~\cite[VII.3.6, Proposition 4]{bourbaki72}.

Four attributes of the chain $(R_\beta)_{\beta<|T|}$ govern the behavior of the formal fibers of~$A$.
First, the intermediate extensions $R_\beta\subseteq R_{\beta+1}$ are $J$-extensions.
The ring~$A$ consequently satisfies every defining condition of a $J$-subring except~\ref{nice1}, and therefore $\frak P \cap A = (0)$.
Second, each $R_\beta$ satisfies~\ref{nice6}. 
The ring~$A$ consequently satisfies~\ref{nice6}, which will help us show 
that $A$ satisfies property~\eqref{main3}.
Third, we control the heights of new elements of the distinguished sets~$\cal Q_{R_\beta}$ for each extension.
The ring~$A$ consequently satisfies~\eqref{main2}.
Fourth, $R_{\beta+1}$ is formed from $R_\beta$ by adjoining generators of a given prime ideal $\fq_\beta$ of $T$, eventually adjoining generators for every prime ideal not contained in an element of some distinguished set or in $\frak P$.
The ring~$A$ consequently satisfies~\eqref{main4} and $\alpha(A,(0))
 = \height \frak P$.
This concludes the overview of the construction.

Before constructing the chain $(R_\beta)_{\beta<|T|}$
we require several set-theoretic bookkeeping preliminaries.
The first of these is to choose coset representatives and generators of prime ideals to be added at each extension.
The second is to explain how the heights of the elements of the distinguished set
can be chosen in the correct proportions.
This second preliminary will be used to ensure that $|\Delta_d|=\lambda_d$
for every~$d$, since the heights effect the dimensions of the formal fibers.
Specifically, the height of an element of the distinguished set is one more
than the dimension of the corresponding formal fiber.

First, let $C_1$ be a set of coset representatives for~$T/\m^2$
and let $C_2 = \bigcup\{\spec_i T\mid 1\leq i\leq\dim T-1\}$. 
Now use Lemma \ref{avoidancelem} in the same way as we did in the proof of Lemma \ref{pickq} to show that the number of height one prime ideals of $T$ is at least $|T/\m|$.
Since $|T|=|T/\m|$, it follows that $C_1$ and $C_2$ both have cardinality~$|T|$.
We may therefore well-order $C_1$ and~$C_2$ to be order isomorphic to~$|T|$.  We well-order these sets so that they do not have a maximal element.
Let $t_\beta$ and $\fq_\beta$ denote the elements of $C_1$ and $C_2$, respectively,
that correspond to~$\beta$.

Controlling the heights is only slightly more complicated.
The problem is to construct a height indicator $\Lambda$
that assumes the value~$d+1$ exactly~$\lambda_d$ times.
This is basic set theory: the equation $|T| = \sum_d\lambda_d$
and the definition of ordinal addition
mean that the set $|T|$ can be partitioned
into $\min(\height\frak P + 1,\dim T - 1)$ subsets,
the $(d+1)$th of which has cardinality~$\lambda_d$.

Let us now construct the chain~$(R_\beta)_{\beta<|T|}$
and the union~$A$ of its elements.
To start, let~$R_0$ be the prime subring of~$T$,
which is necessarily a field because $T$ is equicharacteristic.
Hence the set $\mathcal F_{R_0}$ is empty
and $R_0$ has a $J$-subring structure
with height indicator~$\Lambda$
and $\cal Q_{R_0} = \varphi_{R_0} = \varnothing$.

The nontrivial step of the recursion is the case where $\beta=\gamma+1$ is a successor ordinal. In this case, if $\fq_\gamma$ is contained in $\frak P$ or in some element of~$\cal Q_{R_\gamma}$,
then define $R_{\beta}$ in the following way.  First let $R'$ be the $J$-extension obtained from Lemma \ref{cosets} so that $t_{\gamma} + \m^2$ is in the image of the map $R' \to T/\m^2$.  Then let $R_{\beta}$ be the $J$-extension of $R'$ obtained from Lemma \ref{ideals} so that $R_{\beta}$ satisfies \ref{nice6} and so that for every finitely generated ideal~$\fa$ of~$R_\beta$ we have that $\fa T\cap R_\beta=\fa$.  If $\fq_\gamma$ is not contained in $\frak P$ or in some element of~$\cal Q_{R_\gamma}$ then
use Lemma~\ref{pretheorem} to construct a $J$-extension $R_\beta$ of $R_\gamma$ 
satisfying \ref{nice6} such that
$t_\gamma+\m^2$ is in the image of the map $R_\beta\to T/\m^2$,
for every finitely generated ideal~$\fa$ of~$R_\beta$ we have that $\fa T\cap R_\beta=\fa$,
and either $(\fq_\gamma\cap R_\beta)T = \fq_\gamma$ or $\fq_{\gamma} \subseteq \fq$ for some $\fq \in \cal Q_{R_{\beta}}$.
Invoking Lemma~\ref{pretheorem} requires that $T$ contains a field of cardinality~$|T|$,
but since $T$ is equicharacteristic it contains a coefficient field
isomorphic to $T/\m$, which has the same cardinality as~$T$.

If $\beta$ is a limit ordinal, then take $R_\beta=\bigcup_{\gamma<\beta}R_\gamma$.
Finally, let $A=\bigcup_{\beta<|T|} R_\beta$.
The construction is thereby concluded.

We now catalog the properties of~$A$.
It is clear that $A$ satisfies~\ref{nice6}. 
Following the proof of Lemma \ref{ideals} the system $(R_\beta)_{\beta < |T|}$ satisfies the hypotheses of Lemma~\ref{unions}.
Therefore $A$ satisfies~\ref{nice2},~\ref{nice3} and~\ref{nice4}
with $\cal F_A = \bigcup_{\beta<|T|} \cal F_{R_\beta}$
and $\cal Q_A = \bigcup_{\beta<|T|} \cal Q_{R_\beta}$.

The fact that $\widehat A = T$ is a consequence of Lemma~\ref{machine}.
Specifically,
if $\fa=(a_1,\ldots,a_n)A$ is a finitely generated ideal of $A$
and $c\in\fa T\cap A$, then there is some $\beta < |T|$ where $\beta$ is a successor ordinal and 
such that $a_1,\ldots,a_n,c\in R_\beta$.
Then $c\in(a_1,\ldots,a_n)T\cap R_\beta=(a_1,\ldots,a_n)R_\beta\subset\fa$,
and so we have that $\fa T\cap A=\fa$.
The map $A\rightarrow T\slash\m^2$ is surjective by construction,
and 
$A$ is quasi-local with maximal ideal $\m\cap A$ since each $R_{\beta}$ is quasi-local with maximal ideal $R_{\beta} \cap \m$.
Using Lemma~\ref{machine} we see that $A$ is Noetherian and $\widehat A=T$,
and furthermore $A$ is a UFD because $T$ is.

We will now show that $A$ satisfies the remaining four conditions. 

The ring~$A$ inherits an ordering function~$\varphi_A$ in the obvious way,
as in the proof of Lemma~\ref{unions}.
Since we well-ordered our sets so that there was no maximal element, $|\cal F_{R_\beta}| < |T|$.
It follows that, for every~$\beta < |T|$, 
the map $\varphi_{R_\beta}$ is not surjective.
Furthermore, $|A|=|T|$ because $|T|=|T/\m|=|T/\m^2|$
and the map $A\to T/\m^2$ is surjective.
An application of Lemma~\ref{avoidancelem}
then shows that $|\cal F_A| = |\spec_1 A| = |T|$.
The map $\varphi_A$ effects a bijection between $\cal F_A$
and an initial segment of~$|T|$, and by what we just observed,
this initial segment has cardinality~$|T|$.
Hence the initial segment equals~$|T|$;
this is the minimality property that characterizes cardinal
numbers among the ordinal numbers.
Therefore, for every $\beta<|T|$,
there is exactly one $pT\in\cal F_A$ with $\varphi_A(pT)=\beta$
and $\height(\fq_{pT}) = \Lambda(\beta)$.
So $|\Delta_d| = \lambda_d$ for each~$d$.

Recall that, for each $\p\in\spec A$, the elements of the formal fiber of $\p$ are in one-to-one correspondence with the prime ideals of $T\otimes_A\kappa(\p)$. 
Let $\p$ be a prime ideal of $A$; since $T$ is faithfully flat over $A$, some $\fq\in\spec T$ lies over $\p$ and $\height\fq\geq\height\p$. 
First suppose that $\height\p=1$, and so $\p=pA$ for some nonzero prime element of $A$.
It follows from conditions~\ref{nice4} and~\ref{nice6}
that $\fq_{pT}$ is the only element of $\cal Q_A$ contained in the formal fiber of $pA$.
Furthermore, if $\fq\in\spec T$ is not contained in $\fq_{pT}$,
then $\fq$ cannot be contained in the formal fiber of $pA$.
To see this, first suppose $\fq \subseteq \fq_{p'T}$ for some $\fq_{p'T} \in \cal{Q}_A$ with $\fq_{p'T} \neq \fq_{pT}$.  Then $\fq \cap A \subseteq \fq_{p'T} \cap A = p'A \neq pA$.  
If $\fq \subseteq \frak{P}$, then $\fq \cap A = (0) \neq pA$.
If, on the other hand, $\fq \not\subseteq \frak{P}$ and no element of $\cal Q_A$ contains $\fq$, then
by our construction we have that $\fq=(\fq\cap A)T\neq pT$
and so $\fq\cap A\neq pA$.
Thus we have shown that the formal fiber ring $T\otimes_A\kappa(pA)$ is local with maximal ideal corresponding to $\fq_{pT}$. It follows that the ring $T\otimes_A\kappa(pA)$ is isomorphic to $T/pT$ localized at the image of the ideal $\fq_{pT}$ in the ring $T/pT$.  Since we chose $\fq_{pT}$ so that its image in $T/pT$ is in the regular locus of $T/pT$, it follows that $T\otimes_A\kappa(pA)$ is a regular local ring.  Every prime ideal of $T$ containing $(pA)T=pT$ and contained in $\fq_{pT}$ is the formal fiber of $pA$, so $\alpha(A,pA)=\dim(T\otimes_A\kappa(pA))=\height\fq_{pT}-\height(pT)$.  By construction, there are $\lambda_d$ elements of $\cal Q_A$ that have height $d + 1$.  Note that, if $\height \fq_{pT} = d + 1$, then $\alpha(A,pA) = (d + 1) - 1 = d$.  It follows that there are $\lambda_d$ elements $pA$ of $\cal Q_A$ satisfying $\alpha(A,pA) = d$.

Now let $\height\p\geq 2$. If $\fq\in\spec T$ lies over $\p$, then $\fq$ is not contained in any element of $\cal Q_A$, and $\fq$ is not contained in $\frak{P}$, so $\fq=(\fq\cap A)T=\p T$. Thus, $\p T$ is the only prime ideal contained in the formal fiber of $\p$, and so $T\otimes_A\kappa(\p)$ is a zero-dimensional domain and hence a field.

Finally, we consider $\p=(0)$. By construction, $A \cap \frak{P} = (0)$ and so $\alpha(A,(0))\geq \height \frak{P}$.   Recall that, for $\fq_{pA} \in \cal Q_A$, we have $\height \fq_{pA} \leq \height \frak{P} + 1$.  So if $\fq$ is a prime ideal of $T$ with height greater than that of $\height \frak{P}$, then $\fq \not\subseteq \frak{P}$ and either $\fq \in \cal Q_A$ or $\fq$ is not contained in any distinguished prime ideal of $A$. It follows by our construction that $\fq \cap A \neq (0)$ and so $\alpha(A,(0)) = \height \frak{P}$.
\end{proof}


\begin{remark} 
The formal fibers of the ring $A$ in Theorem \ref{mainTheorem} are completely understood.  The generic formal fiber of the ring~$A$ constructed in Theorem~\ref{mainTheorem}
consists of the prime ideals of~$T$ contained in~$\frak P$
and the prime ideals contained in $\fq_{pT} \in \cal Q_A$ that do not contain $p$.  If $pA \in \spec_1 A$, then the formal fiber ring of $pA$ is local with maximal ideal $\fq_{pT}$.  In particular, the formal fiber of $pA$ is all prime ideals contained in $\fq_{pT}$ that contain $p$.  Finally, if $\p$ is a prime ideal of $A$ with height greater than one, then the formal fiber of $A$ at $\p$ is $\{ \p T \}$.
\end{remark}

\begin{example}
If $T=\mathbb C [[w,x,y,z]]/(x^2 + y^3 + z^5)$, then there exists a local UFD $A$ such that $\widehat{A}= T$, $\alpha(A,(0))=1$, and $\alpha(A,pA)=1$ for every $pA\in\spec_1 A.$
\end{example}

\section{The Excellent Case}

In this section we strengthen the results of the previous one to obtain
excellent rings.
Recall that by the Cohen Structure Theorem, a complete equicharacteristic local ring
is regular if and only if it is a power series ring in finitely many variables with coefficients in a field.
\begin{thm} \label{excellence}
Let $T$ be a complete equicharacteristic regular local ring of characteristic zero with $3\leq\dim T$ and $|T|=|T/\m|$.
Let $\frak P$ be a nonmaximal prime ideal of~$T$,
and let $\{\lambda_d\mid 0\leq d\leq\min(\height\frak P,\dim T - 2)\}$ be a set of cardinal numbers such that
$\sum_d \lambda_d = |T|$.

Then there exists an excellent regular local ring $A$ such that $\widehat A =T$ and the following are satisfied:
\begin{enumerate}
\item 
$\frak P\cap A=(0)$ and
$\alpha(A,(0))=\height\mathfrak{P}$,
\item 
for each $d$ the set $\Delta_d=\{pA\in\spec_1 A\mid\alpha(A,pA)=d\}$
has cardinality $|\Delta_d|=\lambda_d$,
\item $T\otimes_A\kappa(pA)$ is a regular local ring for every $pA\in\spec_1 A$, and
\item $T\otimes_A\kappa(\p)$ is a field for every $\p\in\spec A$ with $\height\p\geq 2$.
\end{enumerate}
\end{thm}

\begin{proof} 
First construct a subring $A$ of $T$ as in Theorem \ref{mainTheorem}.  Since $T$ is a regular local ring and $\widehat{A} = T$, $A$ is a regular local ring.  It remains to show is that $A$ is excellent.

We first show that the formal fibers of $A$ are geometrically regular. Let $\p\in\spec A$; then we must show $T\otimes_A L$ is a regular ring for every finite field extension $L$ of $\kappa(\p)$. It is enough to consider purely inseparable extensions (see Remark 1.3 of \cite{rotthaus97}), and so, since $A$ contains the rationals, $\kappa(\p)$ has characteristic 0 and we need only show that $T\otimes_A\kappa(\p)$ is regular. If $\p$ is a nonzero prime ideal of $A$, then $T\otimes_A\kappa(\p)$ is a regular local ring or a field, and so we now consider $T\otimes_A\kappa((0))$.

Let $\fq\otimes_A\kappa((0))$ be in the formal fiber of $(0)$. The localization of $T\otimes_A\kappa((0))$ at $\fq\otimes_A\kappa((0))$ is isomorphic to $T_{\fq}$, which is regular since $T$ is a regular local ring. Then the localization of $T\otimes_A\kappa((0))$ at each of its prime ideals is a regular local ring, and so $T\otimes_A\kappa((0))$ is regular. $A$ is formally equidimensional since $T$ is a domain, and so $A$ is universally catenary~\cite[Theorem 31.6]{matsumura86} and hence is our desired excellent ring.
\end{proof}

Using Theorem~\ref{excellence} we obtain
the following example, which answers the original question
of Heinzer, Rotthaus, and Sally.
Since every uncountable cardinal number is realized
as the cardinality of a field, it follows that the set~$\Delta$
in their question can be made to have any uncountable cardinality.

\begin{example}\label{exccor} 
Let $k$ be an uncountable field of characteristic zero and let $T=k[[x_1,x_2,\dots,x_n]]$ with $n=\dim T\geq 3$.  Let $d$ be an integer with $0 \leq d \leq n - 2$.
Then there exists an excellent regular local ring $A$ such that $\widehat A = T$ 
and the set $\Delta=\{pA\in\spec_1 A\mid\alpha(A,(0))=\alpha(A,pA) = d\}$ is equal to the set of all height one prime ideals of $A$.  In particular, because the ring $A$ we constructed is such that $|A| = |T| = |\spec_1A|$, we have that $|\Delta| = |T|$.
\end{example}

Finally, we provide an example illustrating that we can construct an excellent regular local ring such that the number of height one prime ideals whose formal fibers are of a specific dimension is prescribed.

\begin{example}\label{ex} 
Let $T=\mathbb{C}[[x_1,x_2, x_3,x_4,x_5]]$. Then there exists an excellent regular local ring $A$ such that $\widehat A = T$ and such that $\alpha(A,(0)) = 4$, $|\Delta_0|=|T|,|\Delta_1|=1,|\Delta_2|=\aleph_0$, and $|\Delta_3|=|T|$ where $\Delta_d=\{pA\in\spec_1 A\mid\alpha(A,pA) = d\}$ for $0\leq d\leq 3$.  
\end{example}


\section*{Acknowledgments}
We would like to thank Williams College for hosting the SMALL REU; the Clare Boothe Luce Program and the National Science Foundation (via grant DMS1347804) for generously supporting our research; and Raymond C. Heitmann for reminding us that the singular locus of an excellent ring is Zariski closed.

\bibliography{references}

\end{document}